\documentclass[12pt]{amsart}

\usepackage{amssymb,amsmath,latexsym,amsthm}
\usepackage[english]{babel}
\usepackage[inner=1.0in,outer=1.0in,bottom=1.0in, top=1.0in]{geometry}

\newtheorem{thm}{Theorem}[section]

\newtheorem{lemma}[thm]{Lemma}

\newtheorem{proposition}[thm]{Proposition}
\theoremstyle{definition}

\newtheorem{definition}[thm]{Definition}
\newtheorem{example}[thm]{Example}
\newtheorem*{mt}{Main Theorem}

\numberwithin{equation}{section}

\newcommand{\lt}{\left}
\newcommand{\rt}{\right}
\newcommand{\bpm}{\begin{pmatrix}}
\newcommand{\epm}{\end{pmatrix}}
\newcommand{\bsm}{\lt(\begin{smallmatrix}}
\newcommand{\esm}{\end{smallmatrix}\rt)}
\newcommand{\beq}{\begin{equation}}
\newcommand{\eeq}{\end{equation}}
\newcommand{\bmat}{\begin{matrix*}}
\newcommand{\emat}{\end{matrix*}}

\newcommand{\Z}{\ensuremath{\mathbb{Z}}}
\newcommand{\N}{\ensuremath{\mathbb{N}}}

\newcommand{\R}{\ensuremath{\mathbb{R}}}

\let\turc\c
\usepackage{etoolbox}
\robustify\turc 
\renewcommand{\c}{\mathfrak{c}}

\newcommand{\vep}{\varepsilon}

\newcommand{\cF}{\ensuremath{\mathcal{F}}}

\newcommand{\inv}{^{-1}}

\usepackage[usenames,dvipsnames]{xcolor}

\usepackage{amsmath, mathtools}

\newcommand{\mr}{\ensuremath{\mathbb R}}

\newcommand{\intR}{\int_{-\infty}^{\infty}}

\begin{document}

\title[Oscillatory integrals with uniformity]{Oscillatory integrals with uniformity in parameters}


\author{Eren Mehmet K{\i}ral}
\address{Eren Mehmet K{\i}ral \\
Department of Mathematics \\
	  Texas A\&M University \\
	  College Station \\
	  TX 77843-3368 \\
		U.S.A.}
		\curraddr{Sophia University \\
Department of Information and Communication Sciences \\
Faculty of Science and Technology \\
7-1 Kioicho, Chiyoda-ku \\
Tokyo 102-8554 \\
Japan}
\email{erenmehmetkiral@protonmail.com}
\urladdr{https://ekiral.github.io/}

\author{Ian Petrow}
\address{Ian Petrow \\
ETH Z\"urich - Departement Mathematik \\
HG G 66.4 \\
R\"amistrasse 101 \\
8092 Z\"urich \\
Switzerland}
\email{ian.petrow@math.ethz.ch}
\urladdr{https://people.math.ethz.ch/{\raise.17ex\hbox{$\scriptstyle\sim$}}petrowi/}

\author{Matthew P. Young}
\address{Matthew P. Young \\
Department of Mathematics \\
	  Texas A\&M University \\
	  College Station \\
	  TX 77843-3368 \\
		U.S.A.}	
\email{myoung@math.tamu.edu}
\urladdr{http://www.math.tamu.edu/{\raise.17ex\hbox{$\scriptstyle\sim$}}myoung/}



\thanks{
 The first author was partially supported by an AMS-Simons travel grant, and thanks RIKEN iTHEMS for access to their office space.\\
 The second author was supported by Swiss National Science Foundation grant PZ00P2\_168164. \\
 This material is based upon work of M.Y. supported by the National Science Foundation under agreement No. DMS-1401008.  Any opinions, findings and conclusions or recommendations expressed in this material are those of the authors and do not necessarily reflect the views of the National Science Foundation.}

\maketitle


\begin{abstract}
We prove a sharp asymptotic formula for certain oscillatory integrals that may be approached using the stationary phase method.  The estimates are uniform in terms of auxiliary parameters, which is
crucial for application in analytic number theory.
\end{abstract}

\bigskip
\section{Introduction}
Exponential integrals occur in many problems in analytic number theory, including moments of $L$-functions, lattice-point counting, and the circle method (see e.g. \cite{GrahamKolesnik}, \cite{HuxleyLatticePoints}, \cite{IwaniecAnalyticBook}).  For some more advanced applications (e.g. in \cite{CI}) it is necessary to develop an asymptotic expansion of multi-dimensional integrals with uniformity in parameters.  For example, in their celebrated paper \cite{CI}, Conrey and Iwaniec are faced with an integral of the form \begin{equation}\label{eg1}\int_{\mr^3} w(x_1, x_2, x_3) e(2 \sqrt{tx_1 x_2 x_3} - x_1 \lambda_1 - x_2 \lambda_2 - x_3 \lambda_3) dx_1 dx_2 dx_3\end{equation}
 for some smooth dyadically supported weight function $w$. In \eqref{eg1} and other intended applications (see also  \cite{KiralYoung5thMoment}, \cite{PetrowYoung}), one is faced with an integral in several variables with several varying parameters, and needs an asymptotic expansion in terms of all of the parameters. In this paper, we  treat such integrals by iteratively applying the stationary phase method in one variable.

As such, this paper concerns the asymptotic behavior of oscillatory integrals of the form
\begin{equation*}
I = I(x_2, \dots, x_d;w,\phi) = \int_{\mathbb{R}} w(x_1, \dots, x_d) e^{i \phi(x_1, \dots, x_d)} dx_1,
\end{equation*}
where $w, \phi$ are smooth functions of $\R^d$ and $w$ has compact support.
We wish to understand the behavior of $I$ as $w$ and $\phi$ vary in certain families defined in terms of derivative bounds.  
One may then apply other oscillatory integral transforms in the auxiliary variables (such as additional stationary phase analyses, Fourier/Mellin transforms, or the integral transforms appearing in the Kuznetsov formula), which is a common technique in analytic number theory.  
The present work hopes to automate
the analysis of $I$ as much as possible, thereby shifting the mental burden off tedious calculations with weight functions so that number theorists can focus on more arithmetic aspects.

Next we discuss some of the existing results in the literature, and why they are unsatisfactory for some applications the authors have encountered in analytic number theory.
The stationary phase method appears in many standard textbooks, e.g. see \cite[Theorem 7.7.5]{Hormander}, \cite[Ch.~VIII Proposition 3]{Stein}, or \cite[Theorem 3.11]{Zworski}.
The method gives
an asymptotic expansion of a Fourier integral of the form $\intR e^{i \lambda \varphi(x)} a(x) dx$, where $a$ is smooth of compact support, under the assumption that $\varphi'(x_0) = 0$ for a unique point $x_0$ in the support of $a$.  
We do not wish to restrict attention to phase functions of the form $\lambda \varphi(x)$.  For instance, one may consider a phase of the form  $\lambda x - x^2$ or
$x - \lambda \log x$.  Although one may sometimes reduce to the case $\lambda \varphi(x)$ by some ad-hoc change of variables, e.g. $x \rightarrow \lambda x$ in the first example above, it is not desirable to  require a pre-processing step.

In order to arrive at the primary results of this paper stated in Section \ref{section:stationaryPhase}, we study the main term resulting from a stationary phase analysis. This main term is given in terms of differential operators applied to $a$, and evaluated at $x_0$ (which is the stationary point, defined implicitly in terms of these auxiliary variables).  These differential operators have coefficients depending on $\varphi$ (with negative powers of $\varphi''(x_0)$, see \cite[(3.4.11)]{Zworski}).  Most of the work in this paper consists of bounding the derivatives of this main term with respect to the remaining variables. 

The paper is organized as follows.  In Section \ref{section:InertFunctions} we define the class of functions of interest to us here, followed by some examples and easy properties.  In Section \ref{section:stationaryPhase} we state our Main Theorem
, which gives an asymptotic formula for $I$ under conditions ensuring the stationary phase method may be applied.  The crucial point here is that we establish derivative bounds on $I$, which are strong enough that one may often fruitfully and easily iterate the stationary phase method.  In Section \ref{section:example} we illustrate this process with an example which was chosen due to its application to some moment problems in the analytic theory of $L$-functions
\cite{CI},
\cite{KiralYoung5thMoment},
\cite{PetrowYoung}.  We give the proof of the Main Theorem in Section \ref{section:proof}.

\section{Inert functions}\label{section:InertFunctions}
  \subsection{Basic Definition}
 We begin with certain families of functions defined by derivative bounds.
Let $\cF$ be an index set and $X=X_T: \cF \to \R_{\geq 1}$ be a function of $T \in \cF$.
\begin{definition}\label{inert}
A family $\{w_T\}_{T\in \cF}$ of smooth  
functions supported on a product of dyadic intervals in $\R_{>0}^d$ is called $X$-inert if for each $j=(j_1,\ldots,j_d) \in \Z_{\geq 0}^d$ we have \begin{equation*}
C(j_1,\ldots,j_d) 
:= \sup_{T \in \cF} \sup_{(x_1, \ldots, x_d) \in \R_{>0}^d} X_T^{-j_1- \cdots -j_d}\left| x_{1}^{j_1} \cdots x_{d}^{j_d} w_T^{(j_1,\ldots,j_d)}(x_1,\ldots,x_d) \right| < \infty.
\end{equation*}
\end{definition}
The notion of $X$-inert measures the uniformity of the ``flatness'' of the functions $w_T$ as we move across the family $\cF$.

    We also remark that the assumption that $w_T$ has support on a product of dyadic intervals is often easily attained, by application of a dyadic partition of unity. We often abbreviate the sequence of constants $C(j_1,\ldots, j_d)$ associated to a family of inert functions by $C_\cF$.

{\bf Convention}. Throughout this paper, constants implied by $\ll$ and $O()$ symbols are uniform with respect to $\mathcal{F}$, and depend only on $C_{\mathcal{F}}$.
On the occasion that an implied constant depends on some additional auxiliary parameter, e.g. $\varepsilon >0$ or $A \geq 1$, we will 
place it as a subscript.
We also use the standard notation $e(x) = e^{2\pi i x}$.

    \subsection{Examples}
    We present several examples of how inert families may be constructed.
    \begin{example}[Dilation] Let $w(x_1, \ldots, x_d)$ be a fixed smooth function that is supported on $[1,2]^d$ and define
    \begin{equation}
        w_{X_1, \ldots, X_d}(x_1, \ldots x_d) = w\Big(\frac{x_1}{X_1} , \ldots, \frac{x_d}{X_d} \Big).
    \end{equation}
    Then with $\mathcal{F}=\{T = (X_1, \ldots, X_d) \in \R_{>0}^d\}$, the family $\{w_T\}_{T \in \cF}$ is $1$-inert.
    \end{example}
\begin{example}[Oscillation]\label{WTexample2} With $w$ as in the previous example we let
\begin{equation}
\label{eq:WTexample2}
 W_{T}(x_1, \dots, x_d) = e^{i \lambda_1 x_1 + \dots + i \lambda_d x_d} w\Big(\frac{x_1}{X_1} , \ldots, \frac{x_d}{X_d} \Big),
\end{equation}
 but now $\cF =\{T = (X_1, \dots, X_d, \lambda_1, \dots, \lambda_d)\}$.  It is easy to see $W_T$ is $X$-inert with $X=X_T = 1 + \max(|\lambda_1| X_1, \dots, |\lambda_d| X_d)$, 
but not $Y$-inert for any $Y=Y_T$ such that $Y_T/X_T \to 0$ as $X_T\to \infty$.
 \end{example}

\begin{example}[Products]
Let $\{w_T\}_{T \in \cF}$ and $\{v_{T'}\}_{T' \in \cF'}$ be $X$ and $Y$-inert families respectively.  Then $\{ w_T \cdot v_{T'}\}_{(T,T') \in \cF \times \cF'}$ is a $\max(X_T,Y_{T'})$-inert family. \end{example}  For instance, in the one-variable case, we have
\begin{equation*}
      |\max(X,Y)^{-j} x^j \frac{d^j}{dx^j} w_T(x) v_{T'}(x) | \leq \sum_{k=0}^{j} \binom{j}{k} C_w(k) C_v(j-k).
   \end{equation*}
   
\begin{example}[Specialization]
Suppose that $\{w_T(x_1, \dots, x_d)\}_{T\in \cF}$ is $X$-inert, supported on $x_i \asymp X_i$, and that we specialize $x_1 = X_1 f( \frac{x_2}{X_2}, \dots,  \frac{x_d}{X_d})$, say, where $f$ is a fixed smooth function.  Let $$W_T(x_2, \dots, x_d) = w_T(X_1 f( \frac{x_2}{X_2}, \dots,  \frac{x_d}{X_d}), x_2, \dots, x_d).$$ Then $\{W_T\}_{T\in \cF}$ is also $X$-inert. 
\end{example}   
One may deduce this quickly from \eqref{eq:Fchainruleformula} below.

  \subsection{Fourier transforms}
    \label{section:InertFourier}
    Under the Fourier transform inert functions behave regularly.  
Suppose that $w_T(x_1, \dots, x_d)$ is $X$-inert and supported on $x_i \asymp X_i$ 
    for each $i$.  Let
    \begin{equation*}
        \widehat{w_T}(t_1, x_2, \dots, x_d) = \intR w_T(x_1, \dots, x_d) e(-x_1 t_1) d x_1
    \end{equation*}
    denote its Fourier transform in the $x_1$-variable.   
    \begin{proposition}
\label{prop:InertFourier}
 Suppose that $\{ w_T: T \in \mathcal{F} \}$ is a family of $X$-inert 
    functions such that $x_1$ is supported on $x_1 \asymp X_1$, and $\{w_{\pm Y_1}: Y_1 \in  (0,\infty)\}$ is a $1$-inert family of functions with support on $\pm t_1 \asymp  Y_1$.
	Then the family $\{X_1^{-1} w_{\pm Y_1}(t_1) \widehat{w_T}(t_1, x_2, \dots, x_d): (T,  \pm Y_1) \in \mathcal{F} \times \pm (0, \infty) \} \}$ is $X$-inert. Furthermore if $Y_1 \gg q^\vep X/X_1$ then for any $A>0$, we have
\[
	X_1\inv w_{\pm Y_1}(t_1) \widehat{w_T}(t_1,x_2, \ldots, x_d) \ll_{\varepsilon, A} q^{-A} 
\]
\end{proposition}
\begin{proof}
    It is a standard fact in Fourier analysis that $\widehat{w_T}$ and its derivatives may be bounded in terms of $X, X_1, \dots, X_d$ by integration by parts.
    Integrating by parts $j_1$ times  gives
    \begin{equation*}
        \frac{\partial^{j_1} X_1^{-1} \widehat{w_T}(t_1, x_2, \dots, x_d) }{\partial t_1^{j_1}} 
        =    \intR \frac{\partial^{j_1}}{\partial x_1^{j_1}} \Big[ w_T(X_1 x_1, \dots, x_d) x_1^{j_1} \Big] e(-x_1 X_1 t_1) \frac{dx_1}{(-t_1)^{j_1}} \ll \frac{X^{j_1}}{|t_1|^{j_1}},
    \end{equation*}
    since the $j_1$-th derivative of the expression in square brackets is $\ll X^{j_1}$, and is supported on $x_1 \asymp 1$.  
    By a slight generalization of this to allow derivatives with respect to $x_2, \dots, x_d$, we see that 
    $X_1^{-1} \widehat{w_T}(t_1, x_2, \dots, x_d)$ satisfies the desired derivative bound that an $X$-inert function is required 
    to have. 
    
    The missing property is that it is not dyadically supported in $t_1$ (with $t_1 > 0$).
    However, we can get around this minor issue by defining
    $$W_{T, Y_1}(t_1, x_2, \dots x_d) = w_{Y_1}(t_1) \widehat{w_T}(t_1, x_2, \dots x_d),$$ where $\{w_{Y_1} : Y_1 > 0 \}$ is 
    a $1$-inert family, supported on $t_1 \asymp Y_1$ (or $-t_1 \asymp Y_1$).  For instance, $w_{Y_1}$ could be part of a dyadic partition of unity.
Now we can claim that $X_1^{-1} W_{T, Y_1}$ forms an $X$-inert family.
    Moreover, by a similar integration by parts argument, we have that
    \begin{equation*}
        W_{T,Y_1}(t_1, x_2, \dots, x_d) \ll_A X_1 \Big(1 + \frac{|t_1| X_1 }{X}\Big)^{-A} \asymp  \Big(1 + \frac{Y_1 X_1}{X}\Big)^{-A},
    \end{equation*}
    giving the final statement of the proposition. \end{proof}

\section{Stationary phase} \label{section:stationaryPhase}

\begin{mt}[Stationary phase]
 \label{thm:exponentialintegral}
 Suppose $w_T$ is $X$-inert in $t_1, \dots t_d$, supported on $t_1 \asymp Z$ and $t_i \asymp X_i$ for $i=2,\dots, d$.  Suppose that on the support of $w_T$, $\phi = \phi_T$ satisfies
\begin{equation}
\label{eq:phiderivatives}
 \frac{\partial^{a_1 + a_2 + \dots + a_d}}{\partial t_1^{a_1} \dots \partial t_d^{a_d}} \phi(t_1, t_2, \dots, t_d) \ll_{C_\mathcal{F}} \frac{Y}{Z^{a_1}} \frac{1}{X_2^{a_2} \dots X_d^{a_d}},
\end{equation}
for all $a_1, \ldots, a_d \in \N$.
Suppose $ \phi''(t_1, t_2, \dots, t_d) \gg \frac{Y}{Z^2}$, (here and below, $\phi'$  and $\phi''$ denote the derivative with respect to $t_1$) for all $t_1, t_2, \dots, t_d$ in the support of $w_T$, and there exists $t_0 \in \mr$ such that $\phi'(t_0) = 0$  (note $t_0$ is necessarily unique). 
Suppose that $Y/X^2 \geq R \geq 1$.  Then 
\begin{equation}
\label{eq:IasymptoticMainThm}
I = \int_{\mathbb{R}} e^{i \phi(t_1, \dots, t_d)} w_T(t_1, \dots, t_d) dt_1 =  \frac{Z}{\sqrt{Y}} e^{i \phi(t_0, t_2, \dots, t_d)} W_T(t_2, \dots, t_d)
 + O_{A}(ZR^{-A}),
\end{equation}
for some $X$-inert family of functions $W_T$, and where $A>0$ may be taken to be arbitrarily large.
The implied constant in \eqref{eq:IasymptoticMainThm} depends only on $A$ and $C_{\mathcal{F}}$.
\end{mt}
The fact that the family $W_T$ is inert in \eqref{eq:IasymptoticMainThm} shows a kind of closure property of the weight functions that appear in stationary phase.  It is therefore well-suited for iterated integrals where the conditions of stationary phase may be applied.

The Main Theorem builds on earlier work of Blomer-Khan-Young \cite{BKY} which obtained an asymptotic formula for $I$ in the one-variable case.  What is new in this paper is the careful analysis of the derivative bounds on the resulting weight function with respect to all the remaining variables $t_2, \dots, t_d$.

To continue this discussion, we synthesize
 Lemma 8.1 and Proposition 8.2 of \cite{BKY} using this language of inert functions, along with some simplified choices of parameters, with the following:
\begin{lemma}[\cite{BKY}]
\label{lemma:exponentialintegral}
 Suppose that $w = w_T(t)$ is a family of $X$-inert functions, 
 with compact support on $[Z, 2Z]$, so that
$w^{(j)}(t) \ll (Z/X)^{-j}$.  Also suppose that $\phi$ is smooth and satisfies $\phi^{(j)}(t) \ll \frac{Y}{Z^j}$ for some $Y/X^2 \geq R \geq 1$ and all $t$ in the support of $w$.  Let
\begin{equation*}
 I = \intR w(t) e^{i \phi(t)} dt.
\end{equation*}
\begin{enumerate}
 \item 
\label{statement1} 
 If $|\phi'(t)| \gg \frac{Y}{Z}$ for all $t$ in the support of $w$, then $I \ll_A Z R^{-A}$ for $A$ arbitrarily large.
 \item If $\phi''(t) \gg \frac{Y}{Z^2}$ for all $t$ in the support of $w$, and there exists $t_0 \in \mr$ such that $\phi'(t_0) = 0$ (note $t_0$ is necessarily unique), then 
 \begin{equation}
 \label{eq:Iformula}
  I = \frac{e^{i \phi(t_0)}}{\sqrt{\phi''(t_0)}}  F_T(t_0) + O_{A}(  Z R^{-A}),
 \end{equation}
where $F_T$ is a family of $X$-inert functions (depending on $A$)  supported on $t_0 \asymp Z$. 
\end{enumerate}
\end{lemma}
In case it is useful in other contexts, we mention that statement \eqref{statement1} only requires $Y/X \geq R$.

The part of the conclusion that $F_T$ is a family of $X$-inert functions  is not explicitly stated that way, but is implicit in \cite[(8.11)]{BKY}.  Here we need to carefully mention that \cite[(8.11)]{BKY} gives bounds on the derivatives of $F_T(y)$ viewing $y$ as an independent variable.  In practice, $y = t_0$ depends on $\phi$ in some way, and so more information is required in order to realize $F_T$ as an inert function with respect to auxilliary variables.    

\section{An example}
\label{section:example}
Suppose $\lambda_1, \lambda_2, \lambda_3 \in \mathbb{R}$, and $w_T(x_1, x_2, x_3)$ is an $X$-inert family supported on $x_i \asymp X_i$ for all $i=1,2,3$.  
Suppose that the family $T \in \mathcal{F}$ may be parameterized by some real number $q \geq 1$, 
where $X \ll_{\varepsilon} q^{\varepsilon}$.  Consider
\begin{equation}
\label{eq:IExampleDef} 
I = \int_{\mr^3} w_T(x_1, x_2, x_3) e(-t x_1 x_2 x_3 + x_1 \lambda_1 + x_2 \lambda_2 + x_3 \lambda_3) dx_1 dx_2 dx_3.
\end{equation}
Suppose that $t >0$ is so that $P:=t X_1 X_2 X_3 \gg q^{\delta}$ for some fixed $\delta > 0$, since otherwise the phase arising from $tx_1 x_2 x_3$ is hardly oscillatory (in this case, $e(-t x_1 x_2 x_3) w_T(x_1, x_2, x_3)$ is $X$-inert with $X \ll_{\varepsilon} q^{\varepsilon}$, so that Proposition \ref{prop:InertFourier} applies).  Also, suppose that $X_i \ll q^{100}$, for each $i=1,2,3$.
To begin, we first locate the $\lambda_i$ into dyadic regions outside of which $I$ is very small.  

Since $\frac{\partial}{\partial x_i} \phi = \lambda_i - \frac{t x_1 x_2 x_3}{x_i}$ (here and throughout this example, our usage of $\phi$ differs by a factor of $2\pi$ from that in the previous section, but this has no significant effect), unless
\begin{equation}
\label{eq:lambdaSize}
\lambda_i \asymp \frac{P}{X_i},
\end{equation}
then $|\frac{\partial}{\partial x_i} \phi| \gg |\lambda_i| + \frac{P}{X_i}$.  In this scenario, we may apply Lemma \ref{lemma:exponentialintegral} part (1) with $Z = X_i$ and $Y = |\lambda_1| X_1 + P$. Since  $Y \gg q^{\delta} \gg X^2 q^{\varepsilon}$ by assumption, we conclude that $I \ll_{A} q^{-A}$ unless \eqref{eq:lambdaSize} holds.

Now suppose that \eqref{eq:lambdaSize} holds.   Viewing $x_2, x_3$ as fixed, change variables $x_1 \rightarrow x_1 \frac{X_2 X_3}{x_2 x_3}$, so that
\begin{equation}
I =  \int_{\mr^3} w_T(x_1, x_2, x_3) e(-t x_1 X_2 X_3 + x_1 \lambda_1 \frac{X_2 X_3}{x_2 x_3} + x_2 \lambda_2 + x_3 \lambda_3) dx_1 dx_2 dx_3,
\end{equation}
where $w_T$ denotes a new $X$-inert family of functions (we do not give a new name to the new family).  The reason to perform this change of variables is to de-linearize the phase so that stationary phase may be applied in either the $x_2$ or $x_3$ variables.  Let us focus on $x_3$ first, where the phase takes the form
$\phi = \lambda_3 x_3 + x_1 \lambda_1 \frac{X_2 X_3}{x_2 x_3}$, so
\begin{equation}
 \frac{\partial}{\partial x_3} \phi = \lambda_3 - x_1 \lambda_1 \frac{X_2 X_3}{x_2 x_3^2}\quad \text{and} \quad  \frac{\partial^2}{\partial x_3^2}\phi  = 2 \frac{x_1\lambda_1 X_2X_3}{x_2 x_3^3}.
\end{equation}
The conditions of the Main Theorem 
 hold with $Y = X_1 \lambda_1 \asymp P$ and $Z = X_3$, and with the stationary point $(x_3)_0 =(\frac{\lambda_1 x_1 X_2 X_3}{\lambda_3 x_2 })^{1/2}$.  For consistency, one may check from \eqref{eq:lambdaSize} that
\begin{equation}
(x_3)_0 \asymp \Big(\frac{\lambda _1 X_1  X_3}{\lambda_3 }\Big)^{1/2} \asymp X_3,
\end{equation}
so that indeed the magnitude of $(x_3)_0$ matches the support of $w_T$.

 We therefore obtain from the Main Theorem 
 that
\begin{equation}
I = \frac{X_3}{\sqrt{P}}
\int_{\mr^2}
e(-t x_1 X_2 X_3) e\Big(x_2 \lambda_2 + 2 \sqrt{\frac{\lambda_1 \lambda_3 x_1 X_2 X_3}{x_2}}\Big) w_T(x_1, x_2) dx_2 dx_1
+ O_{A}(q^{-A}).
\end{equation}
Again, $w_T$ represents a new $X$-inert family of functions.

Now we repeat the process for $x_2$.  The new phase to consider is $\phi = x_2 \lambda_2 + 2 \sqrt{\frac{\lambda_1 \lambda_3 x_1 X_2 X_3}{x_2}}$ which satisfies
\begin{equation}
 \frac{\partial}{\partial x_2} \phi = \lambda_2 - \sqrt{\frac{\lambda_1 \lambda_3 x_1 X_2 X_3}{x_2^3}} \quad \text{ and } \quad  \frac{\partial^2}{\partial x_2^2}\phi =  \frac32 \frac{\sqrt{\lambda_1 \lambda_3 x_1 X_2 X_3}}{x_2^{5/2}} .
\end{equation}
The 
conditions of the Main Theorem 
 hold, with $Z = X_2$ and
\begin{equation}
Y = \sqrt{\lambda_1 \lambda_3 X_1 X_3} \asymp P.
\end{equation}
The stationary point occurs as
\begin{equation}
(x_2)_0 = \Big(\frac{\lambda_1 \lambda_3 x_1 X_2 X_3}{\lambda_2^2}\Big)^{1/3}.
\end{equation}
Thus,
\begin{equation}
I = \frac{X_2 X_3}{P} \intR e(-t x_1 X_2 X_3 +3 (\lambda_1 \lambda_2 \lambda_3 x_1 X_2 X_3)^{1/3} ) w_T(x_1) dx_1 + O_{A}(q^{-A}).
\end{equation}

Finally, we perform stationary phase one final time, on $x_1$.  We have $\phi = -t x_1 X_2 X_3 +3 (\lambda_1 \lambda_2 \lambda_3 x_1 X_2 X_3)^{1/3}$ which satisfies the conditions of Theorem  \ref{thm:exponentialintegral} with $Z = X_1$, and $Y = P$.  The stationary point occurs at
\begin{equation}
(x_1)_0 = \Big(\frac{\lambda_1 \lambda_2 \lambda_3}{t^3 X_2^2 X_3^2} \Big)^{1/2}.
\end{equation}
Similarly to the previous two cases, we have
\begin{equation}
\label{eq:Iasymptotic}
I = \frac{X_1 X_2 X_3}{P^{3/2}} e\Big(2 \sqrt{\frac{\lambda_1 \lambda_2 \lambda_3}{t}}\Big) W_T(\cdot) + O_{A}(q^{-A}).
\end{equation}
Here the notation $W_T(\cdot)$ denotes an inert function after specializing the variables in terms of the ambient parameters $X_i, \lambda_i, t$.

It is an important feature of the above analysis that if the original inert function appearing in \eqref{eq:IExampleDef} is $X$-inert in terms of additional variables (e.g. the $\lambda_i$), then the resulting inert function in \eqref{eq:Iasymptotic} remains $X$-inert with respect to these variables.

An easy modification of the line of reasoning presented in this section can be used to analyze \eqref{eg1}.  Indeed, the astute reader may notice that the Conrey-Iwaniec integral \eqref{eg1} is of the same form as a threefold Fourier transform of \eqref{eq:Iasymptotic} with the substitution $t \leftrightarrow 1/t$.

\section{Proof of the Main Theorem
}\label{section:proof}
Our proof of the Main Theorem 
proceeds gradually.  As a first step, we will show
\begin{lemma}
\label{lemma:statphase}
Suppose $w_T$ is $X$-inert in $t_1, \dots t_d$, 
supported on $t_1 \asymp Z$ and $t_i \asymp X_i$ for $i=2,\dots, d$, 
and $\phi$ satisfies 
\begin{equation}
\label{eq:phiderivatives2}
 \frac{\partial^{a_1 + a_2 + \dots + a_d}}{\partial t_1^{a_1} \dots \partial t_d^{a_d}} \phi(t_1, t_2, \dots, t_d) \ll_{C_{\mathcal{F}}} \frac{Y}{Z^{a_1}} \frac{X^{a_2 + \dots + a_d}}{X_2^{a_2} \dots X_d^{a_d}}
\end{equation}
on the support of $w_T$.
Assume
the conditions in Lemma \ref{lemma:exponentialintegral} part (2) hold for $t=t_1$ (uniformly in $t_2, \dots, t_d$), and that $t_0 $ satisfies 
\begin{equation}
\label{eq:t0derivativeproperty}
 \frac{\partial^{b_2 + \dots + b_d} }{\partial t_2^{b_2} \dots \partial t_d^{b_d} } t_0 \ll_{C_{\mathcal{F}}} \frac{t_0}{X_2^{b_2} \dots X_d^{b_d}},
\end{equation}
for $t_0 \asymp Z$ (that is, $\frac{1}{Z} t_0$ is $1$-inert).
  Then
  \begin{equation}
   I = \frac{e^{i \phi(t_0,t_2, \ldots, t_d)}}{\sqrt{\phi''(t_0, t_2, \dots, t_d)}} F_T(t_0, t_2, \dots, t_d) + O_{A}(Z R^{-A}),
  \end{equation}
where
 $F_T = F_T(t_0, t_2, \dots, t_d)$ is $X$-inert in $t_2, \dots, t_d$.
\end{lemma}
Lemma \ref{lemma:statphase} differs from the Main Theorem 
in a few ways.  The assumption \eqref{eq:phiderivatives2} is slightly weaker than \eqref{eq:phiderivatives} because of the presence of $X$ on the right hand side of \eqref{eq:phiderivatives2}.  Moreover, Lemma \ref{lemma:statphase} contains an additional assumption \eqref{eq:t0derivativeproperty} on the behavior of the function $t_0$ implicitly defined by $\phi'(t_0, t_2, \dots, t_d) = 0$.
Finally, the main term in \eqref{eq:IasymptoticMainThm} is simplified in that $(\phi''(t_0, t_2, \dots, t_d))^{-1/2}$ is, in essence, replaced by $ ZY^{-1/2}$.  Before turning to the proof of Lemma \ref{lemma:statphase}, we state some additional lemmas that bridge the gap from Lemma \ref{lemma:statphase} to the Main Theorem. 

A simple yet common situation occurs when $t_0$ is a generalized monomial in the other variables, meaning
\begin{equation}
\label{eq:t0monomialformula}
 t_0 = c  t_2^{\alpha_2} \dots t_d^{\alpha_d},
\end{equation}
where the $\alpha_i$ are fixed real numbers and $c$ is some constant (depending on $T$).  It is easy to see that if $t_0$ 
is of the form 
\eqref{eq:t0monomialformula}, then it satisfies \eqref{eq:t0derivativeproperty}.
The following result shows that $t_0$ satisfies \eqref{eq:t0derivativeproperty} in much greater generality.
\begin{lemma}
\label{lemma:t0isXinert}  
 Let conditions be as in 
 Lemma \ref{lemma:statphase}.
 Then  $\frac{1}{Z}t_0$  
 (with $t_0$ defined implicitly by $\phi'(t_0, t_2, \dots, t_d) = 0$)
 is $X$-inert.  In particular, if $\phi$ satisfies \eqref{eq:phiderivatives2} with $X=1$ (as in the Main Theorem) 
 the assumption \eqref{eq:t0derivativeproperty} may be omitted from the statement of Lemma \ref{lemma:statphase}.
\end{lemma}
The reader may wonder, then, why we have retained the assumption \eqref{eq:t0derivativeproperty} in Lemma \ref{lemma:statphase}.  One reason is that in many important cases, it is easy to verify \eqref{eq:t0derivativeproperty} directly (e.g. when \eqref{eq:t0monomialformula} holds).  Another reason is that our proof of Lemma \ref{lemma:t0isXinert} builds naturally on the proof of Lemma \ref{lemma:statphase}.

Once one knows that $\frac{1}{Z} t_0$ is $X$-inert, it is then not too difficult to absorb $\phi''(t_0, t_2, \dots, t_d)^{-1/2}$ into the weight function.  For this, we have
\begin{lemma}
\label{lemma:Absorbingphi''t0IntoInertFunction}
Let conditions be as in the Main Theorem
.  Then
\begin{equation}
\frac{F_T(t_0, t_2, \dots, t_d)}{\sqrt{\phi''(t_0, t_2, \dots, t_d)}} = \frac{Z}{\sqrt{Y}} W_T(t_2, \dots, t_d),
\end{equation}
for some family of $X$-inert functions $W_T$.
\end{lemma}

Taken together, Lemmas \ref{lemma:statphase}, \ref{lemma:t0isXinert}, and \ref{lemma:Absorbingphi''t0IntoInertFunction} then finish the proof of the Main Theorem
.

\begin{proof}[Proof of Lemma \ref{lemma:statphase}]
The assumptions in place mean that if 
we consider $t_2, \dots, t_d$ as temporarily fixed, then $I$ meets the conditions of Lemma \ref{lemma:exponentialintegral}, part (2), and so \eqref{eq:Iformula} holds.  The bound \cite[(8.11)]{BKY} gives that
\begin{equation*}
 \frac{d^j}{dy^j} F_T(y) \ll (X/Z)^j \asymp (X/y)^j.
\end{equation*}
However, this estimate views $t_2, \dots, t_d$ as fixed, and does not give bounds on the derivatives of $F$ with respect to $t_i$ with $2 \leq i \leq d$.  To go further, we need to extract the origin of $F = F_T$ from \cite{BKY}.
We have
\begin{equation*}
 F(y) = \sum_{n} p_n(y), \qquad p_n(y) = c_n (\phi''(y))^{-n} G_y^{(2n)}(t)\Big\vert_{t=y},
\end{equation*}
where the sum over $n$ is finite (depending only on the desired value of $A$ in \eqref{eq:Iformula}), $c_n$ are certain absolute constants, and
\begin{equation*}
 G_y(t) = G_y(t;t_2, \dots, t_d)= w_T(t, t_2, \dots, t_d) e^{i H(t, y, t_2, \dots, t_d)},
\end{equation*}
where (with $\phi''$ representing the second derivative in the first variable)
\begin{equation*}
 H(t,y, t_2, \dots, t_d) 
 = \phi(t, t_2, \dots, t_d) - \phi(y, t_2, \dots, t_d) - \tfrac12 \phi''(y, t_2, \dots, t_d) (t-y)^2.
\end{equation*}
{\bf Remarks.} 
Within the definition of $p_n$ (and hence $F$), the symbol $y$ is an arbitrary real number in the support of $w_T$.  Within \eqref{eq:Iformula}, we then substitute $y= t_0$, where now $t_0$ is an implicit function of the other variables.  Moreover, this expansion may seen to be equivalent to \cite[(3.4.11)]{Zworski}.

It may aid the reader to summarize the steps of \cite{BKY} leading to the above expression for $F$.  Firstly, Part (1) of Lemma \ref{lemma:exponentialintegral} follows from repeated integration by parts.  It turns out that under the assumptions of Part (2) of Lemma \ref{lemma:exponentialintegral}, one can well-approximate $I$ by a shorter integral around $t_0$ of length $\asymp \frac{Z}{\sqrt{Y}} R^{\varepsilon}$.  The assumed lower bound on $\phi''$ leads to a lower bound on $|\phi'|$, by the mean value theorem.  One then uses the integration by parts bound on the complement of this small neighborhood around $t_0$ to show this part of the integral is $O(R^{-A})$.  Now, to develop the main term, one approximates $\phi(t)$ by $\phi(t_0) + \frac12 \phi''(t_0) (t-t_0)^2 + \cdots$, where the dots represents the cubic and higher terms which in turn are pulled in to the smooth weight function.  Finally, one uses the Fourier inversion formula on an integral of the form $\intR e^{i A (t-t_0)^2} g(t) dt$ where $g$ has controlled derivatives.

From this point on, the proof of Lemma \ref{lemma:statphase} is self-contained.
We write $p_n$ more explicitly as a function of $y, t_2, \dots, t_d$ as
\begin{equation}
\label{eq:Ft0formulaStatPhase}
 p_n(y, t_2, \dots, t_d) =  c_n \Big(\frac{1}{\phi''(y, t_2, \dots, t_d)}\Big)^n \frac{\partial^{2n}}{\partial t^{2n}} 
 G_y(t;t_2, \dots, t_d) \Big\vert_{t=y}.
\end{equation}
We see that $G_y^{(2n)}(t) \vert_{t=y}$ is a sum of scalar multiples of terms of the form
\begin{equation*}
 w_T^{(\nu_0)}(y, t_2, \dots, t_d) H^{(\nu_1)}(y,y, t_2, \dots, t_d) \cdots H^{(\nu_{\ell})}(y, y,t_2, \dots, t_d),
\end{equation*}
where the superscripts refer  to partial differentiation in the first variable, where $\nu_0\geq 0$, $\nu_1,\ldots, \nu_\ell \geq 1$,  and where $\nu_0 + \nu_1+ \dots + \nu_{\ell} = 2n$.  Note that $H^{(\nu)}(y,y,t_2, \dots, t_d) = \phi^{(\nu)}(y, t_2,\dots, t_d)$ for $\nu \geq 3$, and vanishes otherwise.  We therefore deduce that
\begin{equation}
\label{eq:GyDerivativeBound}
G_y^{(2n)}(t) \Big\vert_{t=y} \ll \max_{\nu_0 + \nu_1+ \dots + \nu_{\ell} = 2n} \Big(\frac{X}{Z}\Big)^{\nu_0} \frac{Y^{\ell}}{Z^{\nu_1 + \dots + \nu_{\ell}}} \ll \frac{X^{2n} + Y^{2n/3}}{Z^{2n}},
\end{equation}
with the final inequality
seen as follows.  Since we may assume $\nu_i \geq 3$ for $i \geq 1$, we have $3\ell \leq  \nu_1 + \dots + \nu_{\ell} = 2n-\nu_0$, whence $X^{\nu_0} Y^{\ell} \leq (X/Y^{1/3})^{\nu_0} Y^{2n/3}$, which is acceptable for $X \leq Y^{1/3}$.  On the other hand, if $X \geq Y^{1/3}$, then we use $Y^{\ell} X^{\nu_0} \leq X^{3\ell + \nu_0} \leq X^{2n}$ to obtain the desired bound.  

Let $J_n(y, t_2, \dots, t_d) = G_y^{(2n)}(t, t_2, \dots, t_d) \vert_{t=y}$.
A slight generalization of \eqref{eq:GyDerivativeBound} shows 
\begin{equation*}
 J_n^{(a_1, a_2, \dots, a_d)}(y, t_2, \dots, t_d) \ll
 \frac{X^{2n} + Y^{2n/3}}{Z^{2n}}
 \frac{X^{a_1 + a_2 + \dots + a_d}}{Z^{a_1} X_2^{a_2} \dots X_d^{a_d}}.
\end{equation*}
The meaning of the superscripts on $J_n$ now mean differentiation with respect to the different variables, viewing $y$ as independent from $t_2, \dots, t_d$.

Next we examine $\Phi_n(y, \dots, t_d) := (\phi''(y, t_2, \dots, t_d))^{-n}$.  We claim
\begin{equation}
\label{eq:PhinDerivativeBounds}
 \Phi_n^{(a_1, a_2, \dots, a_d)}(y, t_2, \dots, t_d) \ll 
 (Z^2/Y)^n \frac{1}{Z^{a_1}}
 \frac{X^{ a_2 + \dots + a_d}}{X_2^{a_2} \dots X_d^{a_d}}.
\end{equation}
For this, we first note that an easy induction argument gives
\begin{equation*}
\frac{d^a}{dx^a} \frac{1}{f(x)} = 
\sum_{\substack{j_1 + \dots +  j_a = a}} c_{j_1, \dots, j_a} 
\frac{(f^{(j_1)}(x)) \dots (f^{(j_a)}(x))}{f(x)^{a + 1}}
,
\end{equation*}
for certain  constants $c_{j_1, \dots, j_a}$.  
Next we generalize to multiple variables.  Let ${\bf j}_i$ be a $d$-tuple of nonnegative integers, and let ${\bf a} = (a_1, \dots, a_d)$. 
Then
\begin{equation}
\label{eq:MultivariableFinverseDerivatives}
\frac{\partial^{a_1 + \dots + a_d}}{\partial y^{a_1} \dots \partial t_d^{a_d} } \frac{1}{f(y, t_2, \dots, t_d)} = 
\sum_{{\bf j}_1 +  {\bf j}_2 + \dots + {\bf j}_{{\bf a} \cdot {\bf 1}} = \bf{a}}
c_{{\bf j}_1, \dots {\bf j}_{{\bf a} \cdot {\bf 1}}}
\frac{f^{({\bf j}_1)} \dots f^{({\bf j}_{{\bf a} \cdot {\bf 1} })}}{f^{{\bf a} \cdot {\bf 1} +1}},
\end{equation}
where ${\bf 1}$ is the $1$-vector of length $d$ (so ${\bf a} \cdot {\bf 1} = a_1 + \dots + a_d$), and $c_{{\bf j}_1, \dots, {\bf j}_{{\bf a} \cdot {\bf 1}}}$ are absolute constants.
This may be easily verified by induction.

One may show directly from \eqref{eq:phiderivatives2} that 
\begin{equation}
\label{eq:phi''bounds}
\frac{\partial^{b_1 + \dots + b_d}}{\partial y^{b_1} \dots \partial t_d^{b_d}}
(\phi''(y, t_2, \dots, t_d))^n \ll \Big(\frac{Y}{Z^2}\Big)^n \frac{1}{Z^{b_1}}  \frac{X^{b_2 + \dots + b_d}}{X_2^{b_2} \dots X_d^{b_d}}.
\end{equation}
Combining \eqref{eq:MultivariableFinverseDerivatives} with \eqref{eq:phi''bounds} and that $\phi''(t) \gg Y/Z^2$, we derive that
\begin{equation*}
 \Phi_n^{(a_1,\ldots, a_d)}(y, t_2, \dots, t_d) 
 \\
 \ll 
\sum_{{\bf j}_1 +  {\bf j}_2 + \dots + {\bf j}_{{\bf a} \cdot {\bf 1}} = \bf{a}}
\frac{
\Big(\frac{Y}{Z^2}\Big)^{n({\bf a} \cdot {\bf 1})} 
\prod_{k=2}^{d} \Big(\frac{X}{X_k}\Big)^{{\bf j_1} \cdot {\bf e}_k + \dots + {\bf j}_{{\bf a} \cdot {\bf 1}} \cdot {\bf e}_k}
}
{Z^{{\bf j}_1 \cdot {\bf e}_1 + \dots + {\bf j}_{{\bf a} \cdot {\bf 1}}  \cdot {\bf e}_1} 
\Big(\frac{Y}{Z^2}\Big)^{n({\bf a} \cdot {\bf 1}+1)} },
\end{equation*}
where ${\bf e}_k$ is the $k$th standard basis vector.
This simplifies to give
the claimed \eqref{eq:PhinDerivativeBounds}.

Putting the above bounds together, we derive that
\begin{equation}
\label{eq:pnderivativebound}
p_n^{(a_1, \dots, a_d)}(y, t_2, \dots, t_d) \ll \Big(\Big(\frac{X^2}{Y}\Big)^n + \Big(\frac{1}{Y^{1/3}}\Big)^n \Big) 
\frac{X^{a_1 + \dots + a_d}}{Z^{a_1} X_2^{a_2} \dots X_d^{a_d}}.
\end{equation}
Since $Y/X^2 \geq R \geq 1$, this gives an asymptotic expansion in $n$ as $R \rightarrow \infty$ (as in \cite{BKY}), and each $p_n$ is $X$-inert in all variables.  Therefore, $F$ is also $X$-inert in all variables (again, viewing $y$ as an independent variable).

As a final step we need to incorporate the fact that $t_0$, which is substituted for $y$, is not an independent variable but rather a function of $t_2, \ldots, t_d$.  We may derive the shape of a general derivative of $F$ as follows.  Let ${\bf a} = (a_2, \dots, a_d)$, ${\bf j} = (j_2, \dots, j_d)$, ${\bf k} = (k_2, \dots, k_d)$, and ${\bf b}_i$ be $(d-1)$-tuples of nonnegative integers.
We claim 
\begin{multline}
\label{eq:Fchainruleformula}
 \frac{\partial^{a_2+ \dots + a_d}}{ \partial t_2^{a_2} \dots  \partial t_d^{a_d} } F(t_0, t_2, \dots, t_d)  
 =
\sum_{{\bf j} + {\bf k} \leq {\bf a}}
\thinspace
\sum_{N \leq j_2 + \dots + j_d}
\thinspace
\sum_{{\bf b}_1 + \dots + {\bf b}_N + {\bf k} = {\bf a}}
c_{{\bf j}, {\bf k},  {\bf b}_1, \dots, {\bf b}_N}
\\
\times F^{(j_2 + \dots + j_d, k_2, \dots, k_d)}
t_0^{({\bf b}_1)} \dots t_0^{({\bf b}_N)}
,
\end{multline}
where the condition ${\bf j} + {\bf k} \leq {\bf a}$ is interpreted componentwise (so $j_{\ell} + k_{\ell} \leq a_{\ell}$ for all $\ell$), and the $c_{*}$'s are absolute constants.  
Moreover, we emphasize that the notation $F^{(j_2 + \dots + j_d, k_2, \dots, k_d)}$ here and below represents partial differentation of $F$ with $y$ viewed as an independent variable.
Once one guesses this shape of expression, it is not difficult to verify it using induction.

Using \eqref{eq:Fchainruleformula}, \eqref{eq:pnderivativebound}, and \eqref{eq:t0derivativeproperty}, we derive
\begin{equation*}
 \frac{\partial^{a_2+ \dots + a_d}}{ \partial t_2^{a_2} \dots  \partial t_d^{a_d} } F(t_0, t_2, \dots, t_d) 
 \\
 \ll 
 \max
 \frac{X^{j_2 + \dots + j_d + k_2 + \dots + k_d}}{Z^{j_2 + \dots + j_d} X_2^{k_2} \dots X_d^{k_d}}
 \frac{Z^N}{\prod_{\ell=2}^{d} X_{\ell}^{({\bf b}_1 + \dots + {\bf b}_N ) \cdot {\bf e}_{\ell}} },
\end{equation*}
where the maximum is over ${\bf j} + {\bf k} \leq {\bf a}$, $N \leq j_2 + \dots + j_d$, ${\bf b}_1 + \dots + {\bf b}_N + {\bf k} = {\bf a}$. 
Since $N \leq j_2 + \dots + j_d$, in the $Z$-aspect, the above bound is $\ll 1$ (meaning, the exponent on $Z$ is $\leq 0$).  The power on $X$ is at most $a_2 + \dots + a_d$, and the power of $X_{\ell}$ in the denominator is $a_{\ell}$.  
Hence
\begin{equation*}
\frac{\partial^{a_2+ \dots + a_d}}{ \partial t_2^{a_2} \dots  \partial t_d^{a_d} } F(t_0, t_2, \dots, t_d)  \ll 
\Big(\frac{X}{X_2} \Big)^{a_2} \dots \Big(\frac{X}{X_d} \Big)^{a_d},
\end{equation*}
which is precisely the desired condition to show that $F$ is $X$-inert.  This completes the proof of Lemma \ref{lemma:statphase}.
\end{proof}

\begin{proof}[Proof of Lemma \ref{lemma:t0isXinert}]
 Let $f = \phi'$ (the derivative with respect to the first variable, $t_1$), so $t_0$ is defined implicitly by $f(t_0, t_2, \dots, t_d) = 0$.  Note that \eqref{eq:phiderivatives2} translates to
 \begin{equation*}
  f^{(a_1, a_2, \dots, a_d)}(t_1, t_2, \dots, t_d) \ll \frac{Y}{Z} \frac{1}{Z^{a_1}} 
  \Big(\frac{X}{X_2}\Big)^{a_2} 
  \dots 
  \Big(\frac{X}{X_d}\Big)^{a_d}.
 \end{equation*}
Likewise, the condition $\phi''(t) \gg Y/Z^2$ means
\begin{equation*}
 f^{(1,0,\dots, 0)}(t_1, t_2, \dots, t_d) \gg \frac{Y}{Z^2}.
\end{equation*}
 
 Implicit differentiation gives
 \begin{equation*}
  t_0^{({\bf e}_{j_0})} = - 
  \frac{f^{({\bf e}_{j_0})}}
    {f^{({\bf e}_{1})}},
 \end{equation*}
where $j_0 \in \{ 2, 3, \dots, d \}$, and ${\bf e}_j$ denotes the $j$-th standard basis vector.  From this, we easily deduce 
$t_0^{({\bf e}_{j_0})} \ll Z \frac{X}{X_{j_0}}$,  consistent  with $\frac{1}{Z}t_0$ being $X$-inert.  Now we proceed inductively to treat arbitrary derivatives.  Let ${\bf a} = (a_2, \dots, a_d)$.  We have
\begin{equation*}
 t_0^{({\bf a} + {\bf e}_{j_0})} = - 
 \frac{\partial^{a_2+ \dots + a_d}}{ \partial t_2^{a_2} \dots  \partial t_d^{a_d} } \frac{f^{({\bf e}_{j_0})}(t_0, t_2, \dots, t_d)}
    {f^{({\bf e}_{1})}(t_0, t_2, \dots, t_d)}.
\end{equation*}
As shorthand, let $g = f^{({\bf e}_{j_0})}$, and $h = f^{({\bf e}_{1})}$.  By \eqref{eq:Fchainruleformula}, we have
\begin{multline*}
 \frac{\partial^{a_2+ \dots + a_d}}{ \partial t_2^{a_2} \dots  \partial t_d^{a_d} } \Big(\frac{g}{h}\Big) 
 =
\sum_{{\bf j} + {\bf k} \leq {\bf a}}
\sum_{N \leq j_2 + \dots + j_d}
\sum_{{\bf b}_1 + \dots + {\bf b}_N + {\bf k} = {\bf a}}
c_{{\bf j}_2, \dots, {\bf j}_d, {\bf k}, N, {\bf b}_1, \dots, {\bf b}_N}
\\
\times \Big(\frac{g}{h}\Big)^{(j_2 + \dots + j_d, k_2, \dots, k_d)}
t_0^{({\bf b}_1)} \dots t_0^{({\bf b}_N)}.
\end{multline*}
Note the total ``degree'' of any ${\bf b}_i$ is at most that of ${\bf a}$, so our inductive hypothesis gives the desired bound for these $t_0^{({\bf b}_i)}$.

We claim that
\begin{equation}
\label{eq:goverhderivatives}
 \Big(\frac{g}{h}\Big)^{(\alpha, k_2, \dots, k_d)}
 \ll
 Z \frac{X}{X_{j_0}} \frac{1}{Z^{\alpha}} \Big(\frac{X}{X_2}\Big)^{k_2} \dots \Big(\frac{X}{X_d}\Big)^{k_d}.
\end{equation}
Taking this for granted for a moment, we derive
\begin{multline*}
 t_0^{({\bf a} + {\bf e}_{j_0})} \ll 
 \sum_{{\bf j} + {\bf k} \leq {\bf a}}
\sum_{N \leq j_2 + \dots + j_d}
\sum_{{\bf b}_1 + \dots + {\bf b}_N + {\bf k} = {\bf a}}
\\
Z \frac{X}{X_{j_0}} \frac{1}{Z^{j_2 + \dots + j_d}} \Big(\frac{X}{X_2}\Big)^{k_2} \dots \Big(\frac{X}{X_d}\Big)^{k_d}
Z^N \prod_{\ell = 2}^{d} \Big(\frac{X}{X_{\ell}}\Big)^{({\bf b}_1 + \dots + {\bf b}_N) \cdot {\bf e}_{\ell}},
\end{multline*}
using the inductive hypothesis.  This simplifies as 
\begin{equation*}
t_0^{({\bf a} + {\bf e}_{j_0})} \ll Z \frac{X}{X_{j_0}} \Big(\frac{X}{X_2}\Big)^{a_2} \dots \Big(\frac{X}{X_d}\Big)^{a_d},
\end{equation*}
as desired.

Now we prove the claim \eqref{eq:goverhderivatives}.
The generalized product rule gives
\begin{equation*}
 \Big(\frac{g}{h}\Big)^{({\bf m})} = \sum_{{\bf m}_1 + {\bf m}_2 = {\bf m}} c_{{\bf m}_1, {\bf m}_2} g^{({\bf m}_1)} \Big(\frac{1}{h}\Big)^{({\bf m}_2)}.
\end{equation*}
Meanwhile, the derivatives of $1/h$ are given by \eqref{eq:MultivariableFinverseDerivatives}.  Therefore,
\begin{multline*}
 \Big(\frac{g}{h}\Big)^{(\alpha, k_2, \dots, k_d)}
 \ll
 \sum_{{\bf m}_1 + {\bf m}_2 = (\alpha, k_2, \dots, k_d)} 
 \frac{Y}{Z} \frac{X}{X_{j_0}} \frac{1}{Z^{{\bf m}_1 \cdot {\bf e}_1}}
 \prod_{\ell=2}^{d} \Big(\frac{X}{X_{\ell}}\Big)^{{\bf m}_1 \cdot {\bf e}_{\ell}}
 \\
 \times 
 \sum_{{\bf j}_1  + \dots + {\bf j}_{{\bf m}_2 \cdot {\bf 1}} = {\bf m}_2} 
 \frac{\Big(\frac{Y}{Z^2}\Big)^{{\bf m}_2 \cdot {\bf 1}}}{\Big(\frac{Y}{Z^2}\Big)^{{\bf m}_2 \cdot {\bf 1} + 1}}
 \frac{1}{Z^{({\bf j}_1 + \dots {\bf j}_{{\bf m}_2 \cdot {\bf 1}}) \cdot {\bf e}_1}} 
 \prod_{\ell=2}^{d} 
    \Big(\frac{X}{X_{\ell}}\Big)^{({\bf j}_1 + \dots {\bf j}_{{\bf m}_2 \cdot {\bf 1}}) \cdot {\bf e}_{\ell}}.
\end{multline*}
This simplifies to give the claimed bound.
\end{proof}

\begin{proof}[Proof of Lemma \ref{lemma:Absorbingphi''t0IntoInertFunction}]
Suppose $D \subset \mathbb{R}$ is an open interval, and $f: D \rightarrow \mathbb{R}$ and $g: \mathbb{R}^d \rightarrow D$ are smooth.  It is not hard to show that
\begin{equation}
\frac{\partial^{a_1 + \dots + a_d}}{\partial x_1^{a_1} \dots \partial x_d^{a_d} }
f(g(x_1, \dots, x_d)) 
= 
\sum_{1 \leq k \leq {\bf a} \cdot {\bf 1}}
\sum_{{\bf j}_1 + \dots + {\bf j}_k = {\bf a}} c_{{\bf j}_1, \dots, {\bf j}_k} f^{(k)}(g) g^{({\bf j}_1)} \dots g^{({\bf j}_k)},
\end{equation} 
for certain constants $c_{*}$.  If $f$ is fixed, and $g$ is part of an $X$-inert family of functions, then one may easily deduce that $f \circ g$ is also an $X$-inert family.

In the present context, we take $g = \frac{Z^2}{Y} \phi''(t_0, t_2, \dots t_d)$, which forms an $X$-inert family of functions, by the previous lemmas. Since $\phi'' \asymp \frac{Y}{Z^2}$ by assumption, the image of $g$ is contained in a fixed open interval of positive reals (not including $0$).  We may take $f(u) = u^{-1/2}$, which is smooth on the image of $g$.
\end{proof}

\providecommand{\bysame}{\leavevmode\hbox to3em{\hrulefill}\thinspace}
\providecommand{\MR}{\relax\ifhmode\unskip\space\fi MR }
\providecommand{\MRhref}[2]{%
  \href{http://www.ams.org/mathscinet-getitem?mr=#1}{#2}
}
\providecommand{\href}[2]{#2}

\end{document}